\documentclass[10pt]{amsart}

\usepackage{amsmath}
\usepackage{amssymb}
\usepackage{bm}
\usepackage{graphicx}
\usepackage{psfrag}
\usepackage{color}
\usepackage{xcolor}
\usepackage{listings}

\definecolor{keywordcolor}{rgb}{0.2, 0.2, 0.75}
\definecolor{stringcolor}{rgb}{0.0, 0.5, 0.0}
\definecolor{commentcolor}{rgb}{0.5, 0.5, 0.5}
\definecolor{backgroundcolor}{rgb}{0.95, 0.95, 0.95}
\definecolor{numbercolor}{rgb}{0.3, 0.3, 0.3}

\definecolor{red}{rgb}{1,0,0}
\definecolor{felix}{rgb}{1,0,0}

\definecolor{pink}{rgb}{1,0,0.4}
\definecolor{darren}{rgb}{1,0,0.4}

\definecolor{blue}{rgb}{0,0,1}
\definecolor{neo}{rgb}{0,0,1}

\usepackage{hyperref}
\definecolor{MITred}{RGB}{163,31,52}
\hypersetup{
  colorlinks=true,
  linkcolor=MITred,
  citecolor=MITred,
  urlcolor=MITred
}
\usepackage{url}
\usepackage{algpseudocode}
\usepackage{fancyhdr}
\usepackage{mathtools}
\usepackage{tikz-cd}
\usepackage{xy}
\xyoption{all}
\usepackage{stmaryrd}
\usepackage{calrsfs}
\usepackage{cleveref}

\theoremstyle{plain}
\newtheorem{thm}{Theorem}[section]
\newtheorem{lem}[thm]{Lemma}
\newtheorem{prop}[thm]{Proposition}
\newtheorem{cor}[thm]{Corollary}

\theoremstyle{definition}
\newtheorem{defn}[thm]{Definition}
\newtheorem{exam}[thm]{Example}

\newtheorem{conj}[thm]{Conjecture}

\theoremstyle{remark}

\numberwithin{equation}{section}

\newcommand{\aaa}{\mathbb{A}}
\newcommand{\cc}{\mathbb{C}}
\newcommand{\ff}{\mathbb{F}}
\newcommand{\nn}{\mathbb{N}}

\newcommand{\pp}{\mathbb{P}}

\newcommand{\qq}{\mathbb{Q}}

\newcommand{\rr}{\mathbb{R}}

\newcommand{\zz}{\mathbb{Z}}

\providecommand\ldb{\llbracket}
\providecommand\rdb{\rrbracket}
\newcommand{\gp}{\mathcal{G}}

\keywords{Bi-UF Positive Conjecture, algebraic monogenic semiring, finitely generated algebraic positive semiring, bi-UF property, half-factoriality, bi-HF property}

\subjclass[2020]{Primary 16Y60, 13F15, 13A05; Secondary 11R09, 13G05}

\begin{document}

\mbox{}
\title{The Bi-UF Positive Conjecture for quadratic monogenic semirings and related progress}

\author{Felix Gotti}
\address{Department of Mathematics\\MIT\\Cambridge, MA 02139}
\email{fgotti@mit.edu}

\author{Omar Graia}
\address{CrowdMath\\Cambridge, MA 02139}
\email{omar.graia2@gmail.com}

\author{Darren Han}
\address{CrowdMath\\Cambridge, MA 02139}
\email{darrenh31415@gmail.com}

\author{Hengrui Liang}
\address{CrowdMath\\Cambridge, MA 02139}
\email{spaceblastxy1@gmail.com}

\date{\today}

\begin{abstract}
  A complex semiring is a subset of the complex plane that is closed under the standard addition and multiplication of complex numbers and contains both $0$ and $1$. A complex semiring $S$ is called a bi-UFS if both its additive monoid $(S,+)$ and its multiplicative monoid $(S\setminus \{1\}, \cdot)$ are unique factorization monoids (UFM). The Bi-UF Positive Conjecture states that $\mathbb{N}_0$ is the only subsemiring of the nonnegative cone of the real line that is a bi-UFS. In this paper, we prove that no simple semiring extension of $\nn_0$ by a quadratic algebraic number is a bi-UFS, identifying a natural class of complex semirings satisfying the statement of the Bi-UF Positive Conjecture. We also identify another class of complex semirings satisfying the statement of the Bi-UF Positive Conjecture. Then we extend the statement of the Bi-UF Positive Conjecture by motivated by a structural theorem we established for semidomains whose additive monoid are finite-rank free commutative monoids. Finally, we consider the bi-HF property, which is a relaxed version of the bi-UF property. We prove that $\nn_0$ is the only positive rational semidomain having the bi-HF property, and we provide two methods to construct bi-HFS complex semirings that are distinct from~$\nn_0$.
\end{abstract}

\bigskip

\maketitle

\bigskip
\section{Introduction}
\label{sec:intro}

The statement of the Fundamental Theorem of Arithmetic (FTA) is central not only in number theory but also in commutative algebra, semigroup theory, and several related fields. In factorization theory we study the arithmetic and atomic structure of algebraic objects governed by weaker versions of the FTA. The statement of the FTA has been conceptualized to cast one of the most relevant classes of integral domains, the class consisting of all unique factorization domains (UFD). More general, we say that a cancellative commutative monoid is factorial or a unique factorization monoid (UFM) if every nonunit can be factored into atoms (i.e., irreducible elements) in a unique manner. A cancellative and commutative monoid is atomic if every nonunit can be factored into finitely many atoms, and we say that an atomic monoid is a half-factorial monoid (HFM) if every two factorizations of the same element have the same number of atoms (counting multiplicity). A subset $S$ of the complex field $\mathbb{C}$ is called a \emph{complex semiring} if it contains $0$ and $1$ and is closed under the usual operations of addition and multiplication. Its additive and multiplicative monoids are denoted by $(S,+)$ and $S^*:=(S \setminus\{0\}, \cdot)$, respectively. A monogenic semiring (of characteristic $0$) is a simple semiring extensions $\nn_0[\rho]$ of the prototypical semiring $\nn_0$ by $\rho \in \cc$:
\[
  \nn_0[\rho]=\{p(\rho):p(x)\in \nn_0[x]\}.
\]
To ease the notation, we let $S_\rho$ denote the monogenic semidomain $\nn_0[\rho]$ and we call $S_\rho$ rational, positive, real, or algebraic according as the generator $\rho$ has the corresponding property. 
\smallskip

The phenomenon of non-unique factorization has been largely investigated in the more general contexts of commutative monoids and, more recently, in the context of semidomains. The interest in the atomic structure, ideal theory, and factorization aspects of monogenic semidomains has significantly increased after the simultaneous appearances in 2019 of the papers~\cite{CF19} by Campanini and Facchini and the paper~\cite{CGG20} by Chapman, Gotti, and Gotti: in~\cite{CF19} the authors study factorization and ideal-theoretical aspects of the transcendental monogenic semidomain $\nn_0[x]$, while in~\cite{CGG20} the authors investigate the factorization and arithmetic of lengths of (the additive structure of) the monogenic rational semidomain. A recent investigation of the multiplicative structure of monogenic semidomain from the lens of factorization has been carried out by Deng, Gotti, and Zeng~\cite{DGZ26}. On the other hand, the additive structure of monogenic semidomains has been considered in several papers recently: see~\cite{ABP21,hP23} for the case of rational generators and~\cite{ABLST23,DDGLPVZ26} for the case of algebraic generators (rational examples had first appeared in~\cite[Section~5]{GG18}).
\smallskip

The following characterization of additive factoriality and additive half-factoriality for algebraic monogenic semidomains is due to Correa-Morris and Gotti, and it plays an essential role through this paper.

\begin{thm} \cite[Theorem~5.4]{CG22} \label{prop:additively-factoriality-characterizations}
  Let $\alpha$ be a positive algebraic number with minimal polynomial $m_\alpha(x) \in \qq[x]$. Then the following conditions are equivalent.
  \begin{enumerate}
    \item[(a)] $(S_\alpha,+)$ is a UFM\@.
      \smallskip

    \item[(b)] $(S_\alpha,+)$ is an HFM\@.
      \smallskip

    \item[(c)] $(S_\alpha,+)$ has $\deg  m_\alpha(x)$ atoms.
    \smallskip

    \item[(d)] $\{\alpha^k : k \in \ldb 0,\deg m_\alpha(x) - 1\rdb \, \}$ is the set of atoms of $(S_\alpha,+)$.
  \end{enumerate}
\end{thm}
\smallskip

A complex semiring whose additive and multiplicative monoids are both UFMs (resp., HFMs) is called a bi-UFS (resp., bi-HFS). As the additive and multiplicative monoids of $\nn_0$ are the (additive) free monoid on $\{1\}$ and the free multiplicative monoid on the set of rational primes, respectively, the prototypical semiring $\nn_0$ is a bi-UFS and so a bi-HFS. In their study of classes of bi-atomic semirings, Baeth, Chapman, and the first author~\cite{BCG21} proposed the so-called ``Bi-UF Positive Conjecture", which is one of the primary motivations of this paper.

\begin{conj}[Bi-UF Positive Conjecture]
  The prototypical semiring $\mathbb{N}_0$ is the only positive semiring that is a bi-UFS.
\end{conj}

\smallskip
\subsection{Established Results}

The conjecture is natural because the two factorization structures pull in opposite directions: while additive factoriality makes a positive semiring look like a free commutative monoid, the factoriality of the multiplicative monoid imposes strong divisibility restrictions. In this paper, we identify two classes of positive semirings that satisfy the statement of the Bi-UF Positive Conjecture. We formulate the natural extended version of the the Bi-UF Positive Conjecture to the class of complex semirings. In~\cite{BCG21}, Baeth, Chapman, and Gotti also posed the question of whether~$\nn_0$ is the only complex semiring consisting of nonnegative real numbers that is a bi-HFS, and a counterexample for this question has been provided by Gonzalez, Polo, Rodriguez in~\cite[Example~4.6]{GPR24}. In the second part of this paper, we consider the half-factorial analogue of the UF-Positive Conjecture. First, we quickly argue that $\nn_0$ is the only positive rational semidomain whose additive and multiplicative monoids are both HFMs. Then we devote the section to the construction of two classes of bi-HFS, including a Laurent-polynomial positive semiring distinct from~$\nn_0$.

\medskip
\subsection{The Roadmap}

In Section~\ref{sec:background}, we collect the notation and terminology from factorization theory and semidomain theory needed to make the paper as self-contained as possible for the reader.
\smallskip

In Section~\ref{sec:monogenic semidomains and the Bi-UF Pos Conj}, we prove that Bi-UF Positive Conjecture holds in two classes of positive semidomains, one of them being the class of simple extension semidomains of the form $\nn_0[\alpha]$, where $\alpha$ is a quadratic algebraic number. For this, we use Theorem~\ref{prop:additively-factoriality-characterizations} as it enforces the relation $\alpha^2=m\alpha+n$ with $m,n\in\nn_0$ whenever $(S_\alpha,+)$ is factorial. This relation gives a two-coordinate description of the elements of $S_\alpha$, allowing coefficient comparisons in multiplicative factorizations. We use these comparisons to identify enough multiplicative irreducibles and then construct explicit non-unique factorizations in every possible quadratic case. We conclude Section~\ref{sec:monogenic semidomains and the Bi-UF Pos Conj} proving that the statement of the Bi-UF Positive Conjecture holds also over certain class of monogenic algebraic semidomains.
\smallskip

In Section~\ref{sec:beyond the Bi-UF Pos Conj}, we establish a finite-rank structure theorem: every semidomain whose additive monoid is a free commutative monoid of finite rank is isomorphic to a finite $\nn_0$-span of algebraic numbers inside a number field. Motivated by this result, we extend the Bi-UF Positive Conjecture to the Bi-UF Complex Conjecture (from positive to complex semidomains). The established result allows us to narrow the scope of the Bi-UF Complex Conjecture to complex semidomains with reduced, finite-rank additive monoids. 
\smallskip

Section~\ref{sec:Bi-HF Property} is devoted to the study of the bi-HF property. First, we prove that $\nn_0$ is the only positive rational semidomain that satisfies the bi-HF property (i.e., whose additive and multiplicative monoids are both half-factorial). Then we provide two constructions of bi-HFS distinct from $\nn_0$, one via pullbacks and one via Laurent polynomials.

\bigskip
\section{Background}\label{sec:background}

We collect here the terminology and notation that will be used throughout the paper.

\medskip
\subsection{General Notation} 

Throughout this paper, we let $\zz$, $\qq$, $\rr$, $\aaa$, and $\cc$ denote the set of integers, rational numbers, real numbers, algebraic complex numbers, and complex numbers, respectively. We let $\pp$, $\nn$, and $\nn_0$ denote the set of standard primes, positive integers, and nonnegative integers, respectively. For any $r,s \in \rr$, we set
\[
	\ldb r,s \rdb := \{n \in \zz : r \le n \le s\}.
\]
Observe that $\ldb r,s \rdb$ is empty for any $r,s \in \rr$ with $r > s$. For a subset $S$ of the real line, we set $S_{\ge r} := \{s \in S : s \ge r\}$ and $S_{> r} := \{s \in S : s > r\}$. 
\smallskip

For the sake of simplicity, it is convenient to assume that all monoids we consider in this paper are cancellative and commutative. To make this possible, within the scope of this paper, we define a \emph{monoid} to be a semigroup with an identity element (the usual definition of a monoid) that is also cancellative and commutative.

\medskip
\subsection{Monoids and Factorizations}

Let $M$ be a monoid written multiplicatively. We denote  the group of units of~$M$ by $M^\times$, and we say that $M$ is \emph{reduced} if $M^\times$ is the trivial group. For $r,s \in M$, we say that $s$ \emph{divides} $r$, and write $s \mid_M r$, if $r=st$ for some $t \in M$, while we say that $r$ and $s$ are \emph{associates} if $r \mid_M s$ and $s \mid_M r$. There is an abelian group $\gp(M)$ that satisfies the following property: $\gp(M)$ contains an isomorphic copy of~$M$ and every abelian group containing an isomorphic copy of $M$ also contains an isomorphic copy of $\gp(M)$. The group $\gp(M)$ is often referred to as the \emph{Grothendieck group} of $M$. The \emph{rank} of~$M$ is the rank of the abelian group $\gp(M)$, viewed as a $\zz$-module.
\smallskip

A nonunit $a \in M$ is called an \emph{atom} (or an \emph{irreducible element}) if for any $r,s \in M$ such that $a=rs$ either $r$ or $s$ is a unit. We let $\mathcal{A}(M)$ denote the set of atoms of~$M$. An element of $M$ is called atomic if it is a unit or factors into finitely many irreducibles. Following Cohn~\cite{pC68}, we say that $M$ is atomic if every element of $M$ is atomic. A nonunit $p \in M$ is called a \emph{prime} if for all $r,s \in M$ such that $p \mid_M rs$ either $p \mid_M r$ or $p \mid_M s$. In addition, an element of $M$ is called \emph{factorial} if it is a unit or it factors into finitely many primes. If every element of $M$ is factorial then we call $M$ a \emph{factorial monoid}. One can check that every prime element is an atom, which immediately implies that factorial elements are atomic and therefore every factorial monoid is atomic.
\smallskip

We say that a pair $z := (A,f)$, where $A$ is a finite subset of $\mathcal{A}(M)$ and $f \colon A \to \nn$, is a \emph{factorization} of an element $r \in M$ if $A$ is finite and $\prod_{a \in A} a^{f(a)} = r$. Observe that the monoid $M$ is \emph{atomic} if every nonunit of $M$ admits a factorization. If every nonunit of $M$ admits exactly one factorization then we say that $M$ is a \emph{unique factorization monoid} (UFM). It is not hard to show that the notions of a factorial monoid and a UFM are equivalent, although here we prefer the term ``factorial monoid'' for brevity, especially when using the modifier ``additively''. Given a factorization $z := (A,f)$ of an element of $M$, we set
\[
  |z| := \sum_{a \in A} f(a)
\]
and call $|z|$ the \emph{length} of~$z$. We say that the monoid $M$ is \emph{half-factorial} or a \emph{half-factorial monoid} (HFM) if $M$ is atomic and any two factorizations of the same element of~$M$ have the same length. Observe that every factorial monoid is half-factorial.

\medskip
\subsection{Semidomains}

An additively written monoid $S$ endowed with a multiplicative binary operation is called a \emph{semiring} if the following two conditions hold:
\begin{itemize}
    \item there exists $1 \in S$ such that $1s = s1 = s$ for all $s \in S$ and
    \smallskip
    \item the multiplicative operation distributes over the additive operation.
\end{itemize}
The element $1$ in a semiring is unique and called the \emph{identity element}. Throughout this paper, we tacitly assume that every semiring $S$ is commutative, which means that $rs = sr$ for all $r,s \in S$.
\smallskip

Let $T$ be a semiring. Then the additive monoid we obtain from $T$ by ignoring its multiplication is denoted by $(T,+)$ and called the \emph{additive monoid} of $T$. A subset $S$ of $T$ is called a \emph{subsemiring} of~$T$ if $S$ is closed under both operations of $T$ and contains the elements $0$ and $1$. Let us introduce the notion of a semidomain, the central algebraic structure of this paper.

\begin{defn}
    A subsemiring of an integral domain is called a \emph{semidomain}.
\end{defn}
\noindent Thus, semidomains form a superclass of that consisting of all integral domains. The subsemirings $\nn_0$ and $\nn_0[x]$ of $\zz$ and $\zz[x]$ are not integral domains but play a central role in this paper.
\smallskip

Let $S$ be a semidomain. The subset $S \setminus \{0\}$ of $S$ is a monoid under multiplication: $S$ is denoted by $S^*$ and called the \emph{multiplicative monoid} of~$S$. The \emph{additive rank} of $S$ refers to the rank of its additive monoid. Let us adapt the algebraic properties we introduced earlier for monoids to the setting of semidomains. We say that $S$ is \emph{reduced} (resp., \emph{atomic}, \emph{factorial}, \emph{half-factorial}) if its multiplicative monoid is reduced (resp., atomic, factorial, half-factorial), while we say that $S$ is \emph{additively reduced} (resp., \emph{additively atomic, additively factorial, additively half-factorial}) if its additive monoid is reduced (resp., atomic, factorial, half-factorial). A factorial semidomain is also called a \emph{unique factorization semidomain}. 
We introduce the bi-UF and the bi-HF properties, which are the most relevant properties in the scope of this paper.
\begin{defn}
    Let $S$ be a semidomain. 
    \begin{itemize}
        \item $S$ is called a \emph{bi-UFS} if $S$ is factorial and additively factorial.
        \smallskip

        \item $S$ is called a \emph{bi-HFS} if $S$ is half-factorial and additively half-factorial.
    \end{itemize}
\end{defn}
\noindent When a semidomain is a bi-UFS (resp., bi-HFS), we also say that it satisfies the bi-UF (resp., bi-HF) property.
\smallskip

A subsemiring of the field $\mathbb{C}$ is called a \emph{complex semidomain}. A complex semidomain $S$ is called \emph{algebraic} (resp., \emph{positive}, \emph{rational}) provided that $S$ consists of algebraic (resp., positive, rational) numbers. 
\smallskip

The subsemiring $\mathbb{S}$ of $S$ generated by $1$ is called the \emph{prime} subsemiring of $S$. Observe that $\mathbb{S} = \mathbb{N}_0 \cdot 1$ and also that $\mathbb{S}$ is contained in every subsemiring of $S$. One can readily check that $\mathbb{S}$ is either the prototypical semidomain $\mathbb{N}_0$ or the field $\ff_p$ for some rational prime $p$. As $\cc$ has characteristic $0$, the prime subsemiring of any complex semidomain is $\mathbb{N}_0$. 
\smallskip

Monogenic and finitely generated semidomains play a relevant role in this paper. Assume that the semidomain $S$ is contained in an ambient field $\ff$. Then $S$ is called \emph{finitely generated} if $S$ can be generated by finitely many elements over its prime subsemiring $\mathbb{S}$, which means that there exist $\alpha_1, \dots, \alpha_n \in \ff$ such that $S = \mathbb{S}[\alpha_1, \dots, \alpha_n]$. Monogenic semidomains are the most special case of finitely generated semidomains: $S$ is called \emph{monogenic} if $S$ has the form $\mathbb{S}[\alpha]$ for some $\alpha \in \ff$. For each $\alpha \in \cc$, the \emph{monogenic semidomain generated by} $\alpha$ is
\[
    S_\rho := \nn_0[\rho] = \{p(\rho) : p(x) \in \nn_0[x]\}.
\]
When $\rho$ is algebraic, $m_\rho(x)$ denotes its minimal polynomial over $\qq$. If $d=\deg m_\rho(x)$ and $(S_\rho,+)$ is a free commutative monoid with basis $1,\rho,\dots,\rho^{d-1}$ then every element of $S_\rho$ has a unique additive expression as a nonnegative integer combination of these powers. In particular, in the quadratic additively factorial case, Proposition~\ref{prop:additively-factoriality-characterizations} gives $\mathcal{A}((S_\alpha,+))=\{1,\alpha\}$, so every element of $S_\alpha$ has a unique expression $c+d\alpha$ with $c,d \in \nn_0$.


\bigskip
\section{Finitely Generated Semidomains and the Bi-UF Positive Conjecture}
\label{sec:monogenic semidomains and the Bi-UF Pos Conj}

In this section, we prove that for any positive quadratic algebraic number $\alpha$, the monogenic semidomain~$S_\alpha$ does not satisfy the bi-UF property.

\medskip
\subsection{The Class of Quadratic Monogenic Semidomains}
\label{sec:quadratic monogenic semidomains}

We start by proving the following simple lemma.

\begin{lem}\label{lem:two-identities-from-a-bc}
  Let $\alpha$ be a positive quadratic algebraic number with minimal polynomial $m_\alpha(x) = x^2 - mx - n \in \qq[x]$. If $S_\alpha$ is additively factorial then the following statements hold.
  \begin{enumerate}
    \item $m,n \in \nn_0$ with $n \neq 0$.
      \smallskip

    \item If $d\alpha + c = (d_1\alpha + c_1)(d_2\alpha + c_2)$ for some $c,d \in \nn_0$ and $c_1, c_2, d_1, d_2 \in \nn_0$ then
      \begin{itemize}
        \item $d=md_1d_2+d_1c_2+d_2c_1$ and
          \smallskip

        \item $c=nd_1d_2+c_1c_2$.
      \end{itemize}
  \end{enumerate}
\end{lem}

\begin{proof}
  (1) Since $S_\alpha$ is additively factorial, Proposition~\ref{prop:additively-factoriality-characterizations} ensures that the set of additive atoms of~$S_\alpha$ is $\{1,\alpha\}$, whence $\alpha^2 = m'\alpha + n'$ for some $m',n' \in \nn_0$. Thus, $\alpha$ is a root of $x^2 - m'x - n' \in \zz[x]$ and so the uniqueness of $m_\alpha(x)$ ensures that $m = m' \in \nn_0$ and $n = n' \in \nn$.
  \smallskip

  (2) Observe that $\alpha^2=m\alpha+n$ because $\alpha$ is a root of $x^2 - mx - n$. Therefore
  \begin{align*}
    d\alpha+c
    &= (d_1\alpha+c_1)(d_2\alpha+c_2) = d_1d_2\alpha^2+(d_1c_2+d_2c_1)\alpha+c_1c_2 \\
    &= d_1d_2(m\alpha+n)+(d_1c_2+d_2c_1)\alpha+c_1c_2 = (md_1d_2+d_1c_2+d_2c_1)\alpha+(nd_1d_2+c_1c_2).
  \end{align*}
  As a consequence, we obtain the desired expressions for $d$ and $c$:
  \[
    d=md_1d_2+d_1c_2+d_2c_1
    \quad \text{and} \quad
    c=nd_1d_2+c_1c_2.
  \]
\end{proof}

We first verify that the monogenic semidomain $S_\alpha$ is reduced if $S_{\alpha}$ is additively factorial.

\begin{lem}\label{lem:Salpha-reduced-if-additively-factorial}
  Let $\alpha$ be a quadratic algebraic number. If $S_\alpha$ is additively factorial then $S_\alpha$ is reduced.
\end{lem}

\begin{proof}
  Let $\alpha$ be a quadratic algebraic number such that $S_\alpha$ is additively factorial. It follows from~\cite[Theorem~5.4]{CG22} that $\alpha^2 \in \nn_0 + \nn_0 \alpha$. Thus, we can write $\alpha^2 = m\alpha + n$ for some $m,n \in \nn_0$ with $n \neq 0$. To argue that $1$ is the only unit of $S_\alpha$, write $1 = (a+b\alpha)(c+d\alpha)$ for some $a,b,c,d \in \nn_0$. Then
  \[
    1 = ac + (ad + bc)\alpha + bd(m\alpha + n) = (ad + bc + bdm)\alpha + (ac + bdn).
  \]
  After comparing the coefficients of $\alpha$, we obtain that $ad+bc+mbd=0$. Therefore both equalities $b=d=0$ and $ac=1$ must hold, from which we conclude that $a + b\alpha = c + d \alpha = 1$. Hence $1$ is the only unit of $S_\alpha$.
\end{proof}

Next we prove that for any $p \in \pp$, the multiplicative monoid $S_{\sqrt{p}}^*$ is not factorial.

\begin{prop}\label{prop:sqrt-p-not-UF}
  The semidomain $S_{\sqrt{p}}$ is not factorial for any $p \in \pp$.
\end{prop}

\begin{proof}
  Fix a rational prime $p$ and set $\alpha := \sqrt{p}$. As the minimal polynomial of $\alpha$ is $x^2 - p$, it follows from Proposition~\ref{prop:additively-factoriality-characterizations} that $S_\alpha$ is additively factorial and that
  \[
    S_\alpha = \{m+n\sqrt{p} : m,n \in \nn_0\}.
  \]
  We will show that $S_\alpha$ is not factorial. Notice that $S_\alpha$ is a reduced semidomain in light of Lemma~\ref{lem:Salpha-reduced-if-additively-factorial}. Since $m_\alpha(x)=x^2-p=x^2-0x-p$, Lemma~\ref{lem:two-identities-from-a-bc} gives that, whenever
  \[
    x+y\sqrt{p}=(a+b\sqrt{p})(c+d\sqrt{p})
  \]
  for some $a,b,c,d \in \nn_0$, one has
  \[
    ac+pbd = x \quad \text{and} \quad ad+bc = y.
  \]
  We split the rest of the proof into two cases, based on the parity of the prime $p$.
  \smallskip

  \textsc{Case 1:} $p=2$. In this case, we can write
  \[
    7+7\sqrt{2}=7(1+\sqrt{2})=(3+\sqrt{2})(1+2\sqrt{2}).
  \]
  The preceding identities show that $7$, $1+\sqrt{2}$, $3+\sqrt{2}$, and $1+2\sqrt{2}$ are atoms. Indeed, for $7$, the equations $ad+bc=0$ and $ac+2bd=7$ force $b=d=0$ and $ac=7$, whence one factor is $1$. For $1+\sqrt{2}$, the equation $ac+2bd=1$ forces $bd=0$ and $a=c=1$, whence $b+d=1$. For $3+\sqrt{2}$, the equation $ad+bc=1$ leaves only the cases $(ad,bc)=(1,0)$ and $(0,1)$, both of which force one factor to be $1$. Finally, for $1+2\sqrt{2}$, the equation $ac+2bd=1$ forces $bd=0$ and $a=c=1$, whence $b+d=2$; because $bd=0$, one factor is $1$.
  \smallskip

  \textsc{Case 2:} $p$ is odd. In this case, we can write
  \begin{equation} \label{eq:a-non-factorial-element-odd-prime}
    (1+\sqrt{p})^2=(p+1)+2\sqrt{p}=2\left(\frac{p+1}{2}+\sqrt{p}\right).
  \end{equation}
  Again the identities above show that all factors in~\eqref{eq:a-non-factorial-element-odd-prime} are atoms. Indeed, $2$ is an atom because $ad+bc=0$ and $ac+pbd=2$ force $b=d=0$ and then $ac=2$. The element $1+\sqrt{p}$ is an atom because $ac+pbd=1$ forces $bd=0$ and $a=c=1$, whence $b+d=1$. Finally, if $\frac{p+1}{2}+\sqrt{p}$ factors nontrivially, then $ad+bc=1$; the cases $(ad,bc)=(1,0)$ and $(0,1)$ force one factor to be $1$, since $p \nmid (p+1)/2$.

  As a consequence, for every rational prime $p$, we have produced two factorizations into atoms with different factors. This, along with the fact that the group of units of $S_\alpha$ is trivial, implies that these factorizations are genuinely distinct. Hence $S_\alpha$ is not factorial.
\end{proof}

As a preparation for the proof of the main theorem, we prove the next two propositions.

\begin{prop}\label{prop:atoms-for-quadratic-UF-a}
  Let $\alpha$ be a positive quadratic algebraic number such that $S_\alpha$ is additively factorial. For any $k \in \nn_0$, the element $\alpha + k \in S_\alpha$ is an atom of $S_\alpha^*$ provided that neither of the following conditions holds:
  \begin{itemize}
    \item $\alpha^2 = \alpha + k$;
      \smallskip

    \item $\alpha^2 = n$ for some $n \in \nn$ such that $n \mid k$.
  \end{itemize}
\end{prop}

\begin{proof}
  Write $m_\alpha(x) = x^2-mx-n$ for some $m,n \in \qq$. By Lemma~\ref{lem:two-identities-from-a-bc}, we have $m,n \in \nn_0$, with $n \neq 0$, and we may use the coefficient identities from that lemma. Suppose that $\alpha+k$ admits a factorization in $S_\alpha$, say
  \[
    \alpha+k=(d_1\alpha+c_1)(d_2\alpha+c_2)
  \]
  for some $c_1,c_2,d_1,d_2 \in \nn_0$. Applying Lemma~\ref{lem:two-identities-from-a-bc} with $d=1$ and $c=k$ gives
  \begin{equation} \label{eq:atom-alpha-plus-k-alpha-coefficient}
    1 = md_1d_2+d_1c_2+d_2c_1 \quad \text{and} \quad k = nd_1d_2+c_1c_2.
  \end{equation}
  Since all three summands in the first equality are nonnegative integers, exactly one of them is equal to $1$ and the other two are equal to~$0$.

  If $md_1d_2=1$, then $m=d_1=d_2=1$, while $d_1c_2=d_2c_1=0$ forces $c_1=c_2=0$. Hence the displayed factorization is $\alpha+k=\alpha^2$, which is the first excluded case.

  It remains to consider the case in which $d_1c_2=1$ or $d_2c_1=1$. By symmetry, assume that $d_1c_2=1$. Then $d_1=c_2=1$, and the vanishing of the other two summands gives $d_2c_1=0$ and $md_2=0$. If $d_2=0$ then the second factor is $d_2\alpha+c_2=1$. If $d_2>0$ then $m=0$ and $c_1=0$, so the second equality in~\eqref{eq:atom-alpha-plus-k-alpha-coefficient} gives $k=nd_2$. In this case, $\alpha^2=n$ and $n \mid k$, which is the second excluded case.

  As a result, if neither exceptional condition in the statement holds then every factorization of $\alpha+k$ in $S_\alpha$ has a factor equal to~$1$. Hence $\alpha+k$ is an atom of $S_\alpha$.
\end{proof}

\begin{prop}\label{prop:prime-numbers-are-atoms-for-quadratic-UF-a}
  Let $\alpha$ be a positive quadratic algebraic number such that $S_\alpha$ is additively factorial. Then every $p \in \pp$ is an atom of $S_\alpha$ unless $\alpha = \sqrt{p}$.
\end{prop}

\begin{proof}
  Write $m_\alpha(x)=x^2-mx-n$ for some $m,n \in \qq$. By Lemma~\ref{lem:two-identities-from-a-bc}, we obtain that $m,n \in \nn_0$ with $n \neq 0$, and we may use the coefficient identities from that lemma. Suppose that $p=(d_1\alpha+c_1)(d_2\alpha+c_2)$ for some $c_1,c_2,d_1,d_2 \in \nn_0$. Applying Lemma~\ref{lem:two-identities-from-a-bc} with $d=0$ and $c=p$ gives
  \begin{equation} \label{eq:prime-atom-coefficient-comparison}
    0 = md_1d_2+d_1c_2+d_2c_1 \quad \text{and} \quad p = nd_1d_2+c_1c_2.
  \end{equation}
  The first equality forces $md_1d_2=d_1c_2=d_2c_1=0$.

  If $d_1=0$ or $d_2=0$ then both $d_1$ and $d_2$ must be zero: indeed, if $d_1=0$ and $d_2>0$ then $c_1=0$, contradicting the second equality in~\eqref{eq:prime-atom-coefficient-comparison}; the other case is symmetric. Thus $p=c_1c_2$, and the primality of $p$ forces one factor to be~$1$.

  Now assume that $d_1,d_2>0$. Then the first equality gives $m=0$ and $c_1=c_2=0$, so $p=nd_1d_2$. Since $\alpha$ is quadratic, $n \neq 1$; otherwise $\alpha^2=1$, forcing $\alpha=1$. Hence $n=p$ and $d_1=d_2=1$, which implies that $\alpha^2=p$, that is, $\alpha=\sqrt{p}$.

  Therefore, unless $\alpha=\sqrt{p}$, every factorization of $p$ in $S_\alpha$ has a factor equal to~$1$. Hence $p$ is an atom of $S_\alpha$.
\end{proof}

\begin{thm}\label{thm:Bi-UF-statement-for-quadratic-ROI}
  Let $\alpha$ be a positive quadratic algebraic number such that $S_\alpha$ is additively factorial. Then $S_\alpha$ is not factorial.
\end{thm}

\begin{proof}
  Write $m_\alpha(x)=x^2-mx-n$ for some $m,n \in \qq$. It follows from Lemma~\ref{lem:two-identities-from-a-bc} that $m,n \in \nn_0$ with $n \ge 1$, and $\alpha^2=m\alpha+n$. We now produce, in each case, two distinct factorizations of the same element into irreducibles. We split the rest of the proof into two cases according to the parity of $m$.
  \smallskip

  \textsc{Case 1:} $m$ is odd. In this case, we set
  \[
    r := m^2+4n \quad \text{and} \quad k := \frac{r-m}{2}.
  \]
  Observe that $r$ is an odd integer such that $r \ge 2$, and so $k \in \nn_0$. Moreover, using the expressions for both $r$ and $k$ above, it is not difficult to verify that
  \[
    k^2+n = r \frac{r-2m+1}{4}.
  \]
  In addition, in light of the identity $r-2m+1=(m-1)^2+4n$, the fact that $m$ is odd guarantees that $(r-2m+1)/4 \in \nn_0$. Hence
  \begin{equation} \label{eq:odd-m-nonunique-factorization}
    (\alpha+k)^2 = r\left(\alpha+\frac{k^2+n}{r}\right).
  \end{equation}
  Proposition~\ref{prop:atoms-for-quadratic-UF-a} shows that $\alpha+k$ is an atom: indeed, the only possible exceptional case here is $\alpha^2 = \alpha + k$, which would force $\alpha + k = m\alpha + n$, and so $m=1$ and $n=k$, but then $k=2n$, contrary to $n>0$.

  If $m \neq 1$ then Proposition~\ref{prop:atoms-for-quadratic-UF-a} also shows that the element $\alpha+(k^2+n)/r$ is irreducible. If $m=1$ then $(k^2+n)/r = n$, so the right-hand side of~\eqref{eq:odd-m-nonunique-factorization} is $r\alpha^2$; in this case $\alpha$ is irreducible by Proposition~\ref{prop:atoms-for-quadratic-UF-a}. Finally, since~$m$ is positive, every rational prime divisor of $r$ is irreducible by Proposition~\ref{prop:prime-numbers-are-atoms-for-quadratic-UF-a}. Thus,~\eqref{eq:odd-m-nonunique-factorization} yields two distinct factorizations into irreducibles of the same element.
  \smallskip

  \textsc{Case 2:} $m$ is even. First, assume that $m=0$. Then $\alpha = \sqrt{n}$, with $n$ not a square. If $n=2$ then the result follows from Proposition~\ref{prop:sqrt-p-not-UF}. If $n$ is odd then $n>1$ and
  \begin{equation} \label{eq:odd-n-square-root-factorization}
    (\sqrt{n}+1)^2 = 2\left(\sqrt{n}+\frac{n+1}{2}\right).
  \end{equation}
  Because $n \nmid 1$ and $n \nmid (n+1)/2$, the two nonconstant factors on the two sides of~\eqref{eq:odd-n-square-root-factorization} are irreducibles by Proposition~\ref{prop:atoms-for-quadratic-UF-a}. In addition, the factor $2$ is irreducible by Proposition~\ref{prop:prime-numbers-are-atoms-for-quadratic-UF-a}. Thus, in light of the identity~\eqref{eq:odd-n-square-root-factorization}, we see that $(\sqrt{n}+1)^2$ has two distinct factorizations into irreducibles.

  Still in the case when $m=0$, it remains to consider even $n>2$. Then
  \begin{equation} \label{eq:even-n-square-root-factorization}
    (\sqrt{n}+2)^2=2\left(2\sqrt{n}+\frac{n+4}{2}\right).
  \end{equation}
  Here $\sqrt{n}+2$ is irreducible by Proposition~\ref{prop:atoms-for-quadratic-UF-a}, and $2$ is irreducible by Proposition~\ref{prop:prime-numbers-are-atoms-for-quadratic-UF-a}. Set $B := 2\sqrt{n}+(n+4)/2$. If $B = (a+b\sqrt{n})(c+d\sqrt{n})$ then, after comparing the coefficient of $\sqrt{n}$, we obtain that $ad + bc = 2$. If $bd > 0$, then the constant coefficient is at least $n$, which is greater than $(n+4)/2$ because $n$ is even, nonsquare, and greater than~$2$. Hence, after swapping factors if necessary, one factor is a positive integer dividing~$2$. Consequently, $B$ is either irreducible or twice an irreducible. Splitting~$B$ if necessary, the right-hand side of~\eqref{eq:even-n-square-root-factorization} gives a factorization into irreducibles that is different from the one on the left-hand side.

  Finally, suppose that $m$ is an even positive integer, and take $m_1 \in \nn$ such that $m = 2m_1$. Now set
  \[
    r := m_1^2+n \quad \text{and} \quad k := r-m_1 = m_1^2-m_1+n.
  \]
  Observe that $r>1$ while $k \in \nn_0$. In addition, the following identity holds:
  \begin{equation} \label{eq:even-m-factorization}
    (\alpha+k)^2=r(2\alpha+r-2m_1+1).
  \end{equation}
  By Proposition~\ref{prop:atoms-for-quadratic-UF-a}, we see that the element $\alpha + k$ is irreducible. On the other hand, every rational prime divisor of~$r$ is irreducible by Proposition~\ref{prop:prime-numbers-are-atoms-for-quadratic-UF-a}. If $m_1=1$, then $2\alpha+r-2m_1+1=\alpha^2$, and $\alpha$ is irreducible by Proposition~\ref{prop:atoms-for-quadratic-UF-a}. Hence~\eqref{eq:even-m-factorization} yields the desired two distinct factorizations into irreducibles of the same element.

  Assume now that $m_1>1$, and set $C := 2\alpha + r - 2m_1 + 1$. Take $a,b,c,d \in \nn_0$ such that $C = (a+b\alpha)(c+d\alpha)$. After comparing the coefficient of $\alpha$, we obtain the equality
  \[
    ad + bc + 2m_1bd = 2.
  \]
  From the equality $m_1>1$, we deduce that $bd=0$. After swapping factors if necessary, we can assume that $b=0$ and, therefore, that $ad=2$. Hence $a=1$ or $a=2$. It follows that $C$ is either irreducible or $2(\alpha+\ell)$. In the latter case, $\alpha+\ell$ is irreducible by Proposition~\ref{prop:atoms-for-quadratic-UF-a}. Splitting $C$ if necessary, the right-hand side of~\eqref{eq:even-m-factorization} gives a factorization into irreducibles that is distinct from the factorization $(\alpha+k)^2$.
  \smallskip

  Hence we conclude that the semidomain $S_\alpha$ does not have the UF property.
\end{proof}

As an immediate consequence, we obtain that the statement of the Bi-UF Positive Conjecture holds for the special class of quadratic monogenic semidomains.

\begin{cor}[Bi-UF Positive Conjecture for Quadratic Monogenic Semidomains]\label{cor:main-result}
  There is no positive quadratic monogenic semidomain that is a bi-UFS.
\end{cor}

\medskip
\subsection{Another Class of Non-Bi-UF Monogenic Semidomains}

It was proved in~\cite{DGZ26} that for any $q \in \qq_{>0}$, the monogenic semidomain $S_q$ is factorial if and only if $q \in \nn \cup \nn^{-1}$. Hence throughout this section, we assume that the generator $\alpha$ of $S_\alpha$ is non-rational.
\smallskip

Let $\alpha$ be a non-rational algebraic number with minimal polynomial $m_\alpha(x)$. We let $w_\alpha(x)$ denote the unique primitive polynomial in $\zz[x]$ with positive constant coefficient that is an integer multiple of $m_\alpha(x)$. It turns out that, for any positive non-rational algebraic number $\alpha$ such that $w_\alpha(0)$ is a composite integer, the monogenic semidomain $S_\alpha$ is not factorial.

\begin{thm} \label{thm:multiplicative obstruction for another class of positive semidomains}
  Let $\alpha$ be a positive non-rational algebraic number. If $w_\alpha(0)$ is a composite integer then the monogenic semidomain $S_\alpha$ is not factorial.
\end{thm}

\begin{proof}
  Let $\alpha$ be as in the statement of the theorem and assume that $w_\alpha(0)$ is a composite integer. We first show that $\alpha$ is an atom of $S_\alpha$. Before doing so, observe that $\alpha$ is not a unit of $S_\alpha$. Indeed, if $\alpha h(\alpha)=1$ for some $h(x) \in \nn_0[x]$, then $xh(x)-1$ vanishes at~$\alpha$. Since $w_\alpha(x)$ is primitive, Gauss's lemma implies that $w_\alpha(x)$ divides $xh(x)-1$ in~$\zz[x]$. Comparing constant terms gives $w_\alpha(0) \mid 1$, contradicting that $w_\alpha(0)$ is composite.
  
  Now take $f(x),g(x) \in \nn_0[x]$ such that $\alpha=f(\alpha)g(\alpha)$. We prove that one of the two factors is a unit of $S_\alpha$. Since $\alpha>0$, neither $f(x)$ nor $g(x)$ can be the zero polynomial. Now we consider the following two cases.
  \smallskip

  \textsc{Case 1:} $\alpha>1$. In this case, for every nonconstant polynomial $h(x) \in \nn_0[x]$, one has $h(\alpha) \ge \alpha$. Hence $f$ and $g$ cannot both be nonconstant as otherwise $\alpha=f(\alpha)g(\alpha) \ge \alpha^2 > \alpha$, which is not possible. They also cannot both be constant polynomials because if this were the case then $\alpha = f(\alpha) g(\alpha) = f(0)g(0) \in \nn$, which is not possible becase $\alpha$ is not rational. Thus, after swapping $f$ and~$g$ if necessary, we may assume that $f$ is constant and $g$ is nonconstant. Since $f(\alpha) = f(0) \in \nn$,
  \[
    \alpha = f(\alpha) g(\alpha) = f(0)g(\alpha) \ge f(0)\alpha,
  \]
  which ensures that $1 \ge f(0)$ and so $f(0)=1$. Hence $f(\alpha)=1$, and then one of the two factors in $\alpha = f(\alpha)g(\alpha)$ is a unit.
  \smallskip

  \textsc{Case 2:} $0<\alpha<1$. Let $r$ and $s$ be the orders of the polynomials $f$ and $g$, respectively, and take $F,G \in \nn_0[x]$ such that
  \[
    f(x) = x^rF(x) \quad \text{and} \quad g(x)=x^sG(x).
  \] 
  Observe that the polynomials $F$ and $G$ have order zero, which implies that $F(\alpha) \ge 1$ and $G(\alpha) \ge 1$. Now write
  \[
    \alpha = f(\alpha) g(\alpha) = \alpha^{r+s}F(\alpha)G(\alpha).
  \]
  Note that $r+s \neq 0$ as otherwise $f(\alpha)g(\alpha) = F(\alpha)G(\alpha) \ge 1 > \alpha$. On the other hand, $r+s < 2$: indeed, if $r+s \ge 2$ then $1=\alpha^{r+s-1}F(\alpha)G(\alpha)$, which is not possible as $\alpha$ is not a unit. Hence $r+s=1$. Therefore $F(\alpha)G(\alpha)=1$, so both $F(\alpha)$ and $G(\alpha)$ are units of $S_\alpha$. Thus, either $f(\alpha)=F(\alpha)$ or $g(\alpha) = G(\alpha)$, and so one of the factors in $\alpha = f(\alpha)g(\alpha)$ is a unit.

  Since in both cases, we have obtained that either $f(\alpha)$ or $g(\alpha)$ is a unit of the semidomain $S_\alpha$, we conclude that $\alpha$ is an atom of $S_\alpha$.
  \smallskip

  Now choose $p_0 \in \pp$ such that $p_0 \mid w_\alpha(0)$, and then take $N \in \nn$ such that $\gcd(N, w_\alpha(0))=1$. We write $w_\alpha(x) = \sum_{k=0}^d c_k x^k$ for some $c_0, c_1, \dots, c_d \in \zz$ and $c_0 \in \nn$. Let us prove the following claim.
  \smallskip

  \noindent \textsc{Claim.} There exists a polynomial $f(x) \in \nn_0[x]$ such that
  \[
    f(0) = Nw_\alpha(0), \qquad f(x) - N w_\alpha(x) \in \nn_0[x], \qquad \text{and} \qquad x+p_0 \mid_{\nn_0[x]} f(x).
  \]

  \noindent \textsc{Proof of Claim.} Let us produce a sequence $a_0, \dots, a_{d+1} \in \nn_0$ such that $f(x) := \sum_{i=0}^{d+1} a_i x^i$ satisfies the three given conditions. In order to satisfy the first condition, we set $a_0 := Nw_\alpha(0)$. Suppose that, for some $k \in \ldb 1,d\rdb$, the coefficients $a_0,\dots,a_{k-1}$ have already been chosen so that
  \[
    S_{k-1}:=\sum_{i=0}^{k-1} a_i(-p_0)^i
  \]
  is divisible by~$p_0^k$ and has sign $(-1)^{k-1}$ when nonzero. Let $r_k$ be the least nonnegative residue of $S_{k-1}$ modulo~$p_0^{k+1}$. Since $p_0^k \mid S_{k-1}$, we see that $p_0^k \mid r_k$. Choose $t \in \nn_0$ sufficiently large, and set
  \[
    a_k:=(-1)^{k+1}\frac{r_k}{p_0^k} + p_0t.
  \]
  Then $S_k:=\sum_{i=0}^k a_i(-p_0)^i$ is divisible by~$p_0^{k+1}$. By taking $t$ larger if necessary, we also ensure that $S_k$ has sign $(-1)^k$ and that $a_k \ge \max\{0,N c_k\}$. After this has been done for $k=d$, define
  \[
    a_{d+1} := (-1)^d \frac{S_d}{p_0^{d+1}}.
  \]
  This is a positive integer, and the definition gives $\sum_{i=0}^{d+1} a_i(-p_0)^i=0$. Hence
  \[
      f(-p_0) = a_{d+1}(-p_0)^{d+1} + \sum_{i=0}^d a_i (-p_0)^i = (-1)^d \frac{S_d}{p_0^{d+1}}(-p_0)^{d+1} + \sum_{i=0}^d a_i (-p_0)^i = -S_d + \sum_{i=0}^d a_i (-p_0)^i = 0.
  \] 
  Thus, $x+p_0$ divides $f(x)$ in $\zz[x]$, which is the third condition. Now write
  \[
      \frac{f(x)}{x+p_0} = \sum_{i=0}^d b_i x^i
  \]
  for some $b_0, \dots, b_d \in \zz$. Observe that
  \[
    \sum_{i=0}^{d+1} a_i x^i = f(x) = (x+p_0) \sum_{i=0}^d b_i x^i = (b_dx^{d+1} + p_0b_0) + \sum_{i=1}^d (b_{i-1} + p_0b_i) x^i,
  \]
  so $p_0b_0 = a_0$ while $p_0^{i+1} b_i = p_0^i a_i - p_0^i b_{i-1}$ for every $i \in \ldb 1,d \rdb$. This implies that, for each $k \in \ldb 0,d \rdb$,
  \[
    p_0^{k+1}b_k = p_0^ka_k + (-1)^1 p_0^{k-1} a_{k-1} + (-1)^2 p_0^{k-2} a_{k-2} + \dots + (-1)^{(k-1)} p_0 a_1 + (-1)^k a_0= \sum_{i=0}^k (-1)^{(k-i)} p_0^i a_i,
  \]
  and so
  \[
      b_k = \frac{1}{p_0^{k+1}} (-1)^k \sum_{i=0}^k a_i(-p_0)^i = \frac{(-1)^kS_k}{p_0^{k+1}} \in \zz
  \]
  because of the fact that $p_0^{k+1} \mid S_k$ for every $k \in \ldb 0,d \rdb$. Moreover, for each $k \in \ldb 0,d \rdb$, the fact that $S_k$ and $(-1)^k$ have the same sign ensures that $b_k \in \nn_0$. As a consequence, $f(x)/(x+p_0) \in \nn_0[x]$. Finally, the inequalities $a_k \ge \max\{0,Nc_k\}$ and the equality $a_0 = Nw_\alpha(0)$ imply that $f(x) - Nw_\alpha(x) \in \nn_0[x]$, which is the second condition of the claim. Hence we have established the claim.
  \smallskip
  
  In order to prove that $S_\alpha$ is not factorial, it suffices to argue that $\alpha$ is an atom that is not prime. To argue this, first set $g(x) := f(x)/(x+p_0) \in \nn_0[x]$ so that we can factor $f(\alpha)$ as follows: 
  \[
    f(\alpha) = (\alpha + p_0) g(\alpha).
  \]
  Since $f(x)-N w_\alpha(x)$ has nonnegative coefficients and constant term~$0$, the divisibility relation $\alpha \mid_{S_\alpha} f(\alpha)$ holds. We claim that $\alpha$ divides neither factor on the right-hand side. If $\alpha \mid_{S_\alpha} \alpha + p_0$ then $xq(x)-x-p_0$ vanishes at~$\alpha$ for some polynomial $q(x) \in \nn_0[x]$. Hence $w_\alpha(x)$ divides $xq(x)-x-p_0$ in~$\zz[x]$, forcing $w_\alpha(0) \mid p_0$, which is impossible because $w_\alpha(0)$ is composite and $p_0 \mid w_\alpha(0)$. Similarly, if $\alpha \mid_{S_\alpha} g(\alpha)$ then $g(x)-xq(x)$ vanishes at~$\alpha$ for some polynomial $q(x) \in \nn_0[x]$. Thus, $w_\alpha(x)$ divides $g(x)-xq(x)$ in~$\zz[x]$, so $w_\alpha(0) \mid g(0)$. But $f(0) = p_0 g(0) = Nw_\alpha(0)$ and, therefore, $g(0) = Nw_\alpha(0)/p_0$; the divisibility $w_\alpha(0) \mid Nw_\alpha(0)/p_0$ would force $p_0 \mid N$, contradicting $\gcd(N,w_\alpha(0))=1$.

  As a result, the atom $\alpha$ divides $f(\alpha) = (\alpha + p_0)g(\alpha)$, but it divides neither $\alpha + p_0$ nor $g(\alpha)$. Hence~$\alpha$ is not a prime. Since every atom in a factorial monoid is prime, so $S_\alpha$ is factorial.
\end{proof}

As a consequence of Theorem~\ref{thm:multiplicative obstruction for another class of positive semidomains}, we can identify another class of positive semidomains whose members do not satisfy the bi-UF condition regardless of whether they are additively factorial.

\begin{cor}
  Let $\alpha$ be a positive non-rational algebraic number. If $w_\alpha(0)$ is a composite integer then $S_\alpha$ is not bi-UFS.
\end{cor}

Given the relevance of algebraic integers, we record the following corollary of Theorem~\ref{thm:multiplicative obstruction for another class of positive semidomains}.

\begin{cor} \label{cor:bi-uf for algebraic integers}
  Let $\alpha$ be a positive non-rational algebraic integer. If $|N_{\mathbb Q(\alpha)/\mathbb Q}(\alpha)|$ is composite then~$S_\alpha$ is not bi-UFS.
\end{cor}

We conclude this section highlighting the remaining monogenic semidomains $S_\alpha$ one must consider to provide a full answer to the Bi-UF Positive Conjecture. These are the monogenic semidomains $S_\alpha$ generated by non-rational $\alpha$ that satisfy the following two conditions:
\begin{itemize}
  \item $S_\alpha$ is additively factorial, which means that the only positive coefficient of $m_\alpha(x)$ is its leading coefficient.
  \smallskip

  \item $S_\alpha$ is factorial, which implies that $w_\alpha(0) \in \{1\} \cup \pp$.
\end{itemize}













\bigskip
\section{Beyond the Bi-UF Positive Conjecture}
\label{sec:beyond the Bi-UF Pos Conj}

The bi-UF property seems to be a strong condition. Indeed, at present, $\mathbb{N}_0$ is the only known semidomain that satisfies the bi-UF property. This calls for a broader search. In this direction, the following conjecture, which we refer to as the \emph{Bi-UF Complex Conjecture}, extends the Bi-UF Positive Conjecture from the nonnegative real ray to the complex plane.

\begin{conj}[Bi-UF Complex Conjecture]
  The unique Bi-UF complex semidomain is $\nn_0$.
\end{conj}

For each quadratic monogenic semidomain $S_\alpha$ considered in Theorem~\ref{thm:Bi-UF-statement-for-quadratic-ROI}, the additive monoid $(S_\alpha,+)$ is both factorial and reduced. In general, it is well known and not difficult to argue that every reduced finite-rank factorial monoid is a finite-rank free commutative monoid \cite[Theorem~1.2.2]{GH06}. It turns out that every semidomain whose additive monoid is a finite-rank free commutative monoid is isomorphic to a complex semidomain; more precisely, it is isomorphic to a subsemiring of an algebraic number field whose additive monoid is finitely generated. We conclude this section by establishing this result.

\begin{thm} \label{thm:structural result}
  Let $S$ be a semidomain whose additive monoid is a finite-rank free commutative monoid. Then there exist $\rho_1, \dots, \rho_n \in \aaa$ such that $S \cong \sum_{k=1}^n \nn_0\rho_k$.
\end{thm}

\begin{proof}
  Let $n$ be the rank of the additive monoid of $S$, and take $e_1, \dots, e_n \in S$ such that $(S,+)$ is free on the set $\{e_1, \dots, e_n\}$. After replacing $S$ by an isomorphic copy embedded into an integral domain, we can assume that $S$ itself is a subsemiring of an integral domain $R$. Let
  \[
    \gp_+(S):=\{s-t : s,t \in S\} \subseteq R.
  \]
  Then $\gp_+(S)$ is the Grothendieck group of $(S,+)$ inside $(R,+)$. It is closed under multiplication because, for all $s_1,s_2,t_1,t_2 \in S$,
  \[
    (s_1-t_1)(s_2-t_2)=(s_1s_2+t_1t_2)-(s_1t_2+s_2t_1).
  \]
  After replacing $R$ by its subring $\gp_+(S)$, we may assume that the Grothendieck group of $(S,+)$ is $(R,+)$. In particular, $R$ is a free $\zz$-module with basis $e_1, \dots, e_n$, so $R = \oplus_{i=1}^n \zz e_i$, and we can identify $(S,+)$ with $\sum_{k=1}^n \nn_0 e_k$. Therefore $R$ is a torsion-free $\zz$-module, which guarantees that $R$ has characteristic~$0$.

  Now extend scalars from $\zz$ to $\qq$ by tensoring $R$ with $\qq$. This produces the $\qq$-vector space $V := \qq \otimes_\zz \gp_+(S)$ and embeds $S$ into a $\qq$-vector space as follows:
  \[
    S \hookrightarrow \gp_+(S) = R \stackrel{\varphi}{\hookrightarrow} V := \qq\otimes_\zz R,
  \]
  where $\varphi \colon \gp_+(S) \to V$ is the $\zz$-module homomorphism determined by $\varphi(g) = 1 \otimes g$ for all $g \in \gp_+(S)$. The map $\varphi$ is injective because $\qq$ is the field of fractions of the PID $\zz$ and $\gp_+(S)$ is a torsion-free $\zz$-module.

  It remains to extend the multiplication on $R$ to $V$. The natural map $\mu \colon R \times R \to R$ that induces the multiplication on $R$ is $\zz$-bilinear, so scalar extension gives a $\qq$-bilinear map
  \[
    \widetilde{\mu} \colon V \times V \to V \quad \text{determined by} \quad \widetilde{\mu}(q \otimes r, q' \otimes r')=qq' \otimes rr'
  \]
  for all $q,q' \in \qq$ and $r,r' \in R$. This turns $V$ into a finite-dimensional commutative $\qq$-algebra extending the ring structure on $R$. Since $R$ is an integral domain of characteristic~$0$, the $\qq$-algebra $V$ is also an integral domain: to see this, it suffices to identify $V$ with the localization $(\zz \setminus \{0\})^{-1}R$.

  Every nonzero element $w \in V$ acts injectively on the finite-dimensional $\qq$-vector space~$V$ by multiplication. Hence $w$ also acts surjectively on $V$, and so $V$ is a field. Thus $V$ is a finite field extension of $\qq$, and there exists an embedding $\sigma \colon V \hookrightarrow \cc$. Finally, set $\rho_k := \sigma(1 \otimes e_k)$ for every $k \in \ldb 1,n\rdb$. From the equality $S = \sum_{k=1}^n \nn_0 (1 \otimes e_k)$, we obtain the desired isomorphism
  \[
    S \cong \sigma(S) = \sigma\bigg(\bigoplus_{k=1}^n \nn_0(1 \otimes e_k)\bigg) = \bigoplus_{k=1}^n \nn_0 \sigma(1 \otimes e_k) = \bigoplus_{k=1}^n \nn_0 \rho_k.
  \]
\end{proof}

From the previous theorem, we already obtained some partial progress towards the Bi-UF Complex Conjecture.

\begin{cor}
  Let $S$ be an additively reduced complex semidomain whose additive monoid has finite-rank. If $S$ is a bi-UFS then there is a number field $K$ and algebraic numbers
  $\rho_1,\dots,\rho_n \in K$ such that
  \[
    S \cong \bigoplus_{k=1}^n \nn_0\rho_k.
  \]
  In particular, if $(S,+)$ has rank $1$ then $S \cong \nn_0$.
\end{cor}

\begin{proof}
  Since $S$ is bi-UF, the additive monoid $(S,+)$ is factorial. Because $(S,+)$ is also reduced and has finite rank, it is a finite-rank free commutative monoid by~\cite[Theorem~1.2.2]{GH06}. The preceding theorem now gives the desired representation as a finite $\nn_0$-span inside a number field.

  If the rank is $1$, let $e$ be the additive generator of $S$. Since $1 \in S$, we can write $1=me$ for some $m\in\nn$. Also $e^2=ce$ for some $c\in\nn_0$. Then
  \[
    e=1e=(me)e=m e^2=mc e.
  \]
  By cancellativity, $mc=1$, so $m=c=1$. Hence $e=1$, and therefore $S=\nn_0$.
\end{proof}

Thus, we can significantly narrow the scope of the Bi-UF Complex Conjecture for a natural class of semidomains as Theorem~\ref{thm:structural result} establishes a concrete structural reduction: to prove the Bi-UF Complex Conjecture for the class of complex semidomains with reduced, finite-rank additive monoids, it suffices to demonstrate that the multiplicative monoid of any such finitely generated $\mathbb{N}_{0}$-span is not factorial whenever $n > 1$. Viewed through this finite-rank framework, the results established for positive quadratic monogenic semidomains in Section~\ref{sec:quadratic monogenic semidomains} resolve the foundational $n=2$ case.

\bigskip
\section{On the Bi-HF Property}
\label{sec:Bi-HF Property}

Recall that a semiring is a bi-HFS or satisfies the bi-HF property if both its additive and multiplicative monoids are HFMs. In~\cite{BCG21}, where the Bi-UF Positive Conjecture was originally posed, the authors also asked whether $\nn_0$ is the only positive semiring that satisfies the bi-HF property. The question was answered by Gonzalez, Polo, and Rodriguez~\cite[Example 4.6]{GPR24}, who constructed a rank-$2$ semidomain with the bi-HF property. In this final section, we construct two classes of bi-HFS, one of them based on the construction provided in~\cite{GPR24}.
\smallskip

Let us start by proving that the only positive rational semidomain whose additive and multiplicative monoids are both HFMs is the prototypical semidomain $\nn_0$.

\begin{prop} \label{prop:rational semidomains}
  The only positive rational semidomain that is a bi-HFS is $\nn_0$.
\end{prop}

\begin{proof}
  It suffices to show that if a positive rational semidomain is additively half-factorial then it is $\nn_0$. Let $S$ be a positive rational semidomain that is additively half-factorial. Now suppose, towards a contradiction, that $S \neq \nn_0$. Then there is an element $q \in S \setminus \nn_0$. Write $q=a/b$ for some relatively prime $a,b \in \nn$ with $b>1$. Since $S$ is closed under multiplication, $q^n \in S$ for every $n \in \nn_0$.

  For each $n \in \nn_0$, let $\ell_n$ denote the length of an additive factorization of $q^n$ in~$S$. This is well defined because $S$ is additively half-factorial. From the equality $a q^n = b q^{n+1}$ in the additive monoid of~$S$, and from half-factoriality, we obtain $a\ell_n = b\ell_{n+1}$ for every $n \in \nn_0$. Iterating this equality gives
  \[
    a^k \ell_0 = b^k \ell_k
  \]
  for every $k \in \nn_0$. Since $\gcd(a,b)=1$, it follows that $b^k$ divides $\ell_0$ for every $k \in \nn_0$, which is impossible because $b>1$ and $\ell_0$ is a positive integer. Hence no such $q$ exists, and so $S=\nn_0$.
\end{proof}

\medskip
\subsection{A Class of Bi-HF Semidomains from a Pullback Construction}

The first bi-HF semiring distinct from~$\nn_0$ was constructed in~\cite[Example~4.6]{GPR24}. We proceed to show a pullback construction that generalizes the example given there.
\smallskip



\begin{prop}\label{prop:pullback-bi-HFS}
  Let $R$ be an atomic half-factorial domain, and let $\varphi \colon R \to \zz$ be a unital ring homomorphism. Set
  \[
    S_\varphi := \{0\} \cup \{r \in R : \varphi(r) \in \nn\}.
  \]
  Then $S_\varphi$ is a bi-HFS\@.
\end{prop}

\begin{proof}
  Since $\nn$ is closed under addition and multiplication, $S_\varphi$ is a subsemiring of~$R$, and so it is a semidomain. We first consider its additive monoid. We claim that
  \[
    \mathcal{A}_+(S_\varphi) = \{s \in S_\varphi : \varphi(s) = 1 \}.
  \]
  Indeed, if $s \in S_\varphi$ and $\varphi(s)>1$, then $s=(s-1)+1$ is a decomposition into two nonzero elements of~$S_\varphi$, so $s$ is not an additive atom. Conversely, if $s=t+u$ for some nonzero $t,u \in S_\varphi$, then $\varphi(s)=\varphi(t)+\varphi(u) \ge 2$, and so no element with image~$1$ can be decomposed additively. Thus the displayed equality holds. It follows that every nonzero element $s \in S_\varphi$ has an additive factorization of length $\varphi(s)$ and that every additive factorization of~$s$ has length $\varphi(s)$. Hence $S_\varphi$ is additively half-factorial.

  It remains to prove that $S_\varphi$ is half-factorial. The group of units of $S_\varphi$ is
  \[
    S_\varphi^\times = \{u \in R^\times : \varphi(u)=1\}.
  \]
  Now let $a$ be an atom of~$R$ with $\varphi(a)>0$. If $a=bc$ for some $b,c \in S_\varphi$ then one of $b$ and $c$ is a unit of~$R$. Since its image under~$\varphi$ is a positive integer and the image of its inverse is also an integer, this unit has image~$1$ and is therefore a unit of~$S_\varphi$. Thus $a$ is an atom of~$S_\varphi^*$.

  Since $R$ is atomic, take $s \in S_\varphi^*$. Write $s=u a_1 \cdots a_n$, where $u \in R^\times$ and $a_1,\dots,a_n$ are atoms of~$R$. Since $\varphi(s)>0$, no $\varphi(a_i)$ is zero. For each $i \in \ldb 1,n\rdb$, set $b_i=a_i$ if $\varphi(a_i)>0$ and $b_i=-a_i$ otherwise. Then each $b_i$ is an atom of~$S_\varphi^*$ by the preceding paragraph, and $s=vb_1\cdots b_n$ for some unit $v \in R^\times$ satisfying $\varphi(v)=1$. Thus $v \in S_\varphi^\times$, and so $S_\varphi^*$ is atomic. 
  
  Finally, we show that every atom of~$S_\varphi^*$ is an atom of~$R$. Suppose, towards a contradiction, that $a \in \mathcal{A}(S_\varphi^*)$ and $a=bc$ for some nonunits $b,c \in R$. Since $\varphi(a)>0$, the nonzero integers $\varphi(b)$ and $\varphi(c)$ have the same sign. If they are positive, then $b,c \in S_\varphi$; if they are negative, then $-b,-c \in S_\varphi$. In either case, $a$ factors in $S_\varphi^*$ as a product of two nonunits, a contradiction. Therefore $\mathcal{A}(S_\varphi^*) \subseteq \mathcal{A}(R \setminus \{0\})$. Since $R$ is half-factorial, any two factorizations in $S_\varphi^*$ of the same element must have the same length. Hence $S_\varphi$ is half-factorial.
\end{proof}

\begin{exam}
  For every $m \in \nn$, the semiring
  \[
    S_m := \{0\} \cup \{f \in \zz[x_1,\dots,x_m] : f(0,\dots,0)>0\}
  \]
  is a bi-HFS\@. Indeed, the polynomial ring $\zz[x_1,\dots,x_m]$ is a UFD and therefore an atomic half-factorial domain, and the constant-term map $f \mapsto f(0,\dots,0)$ is a unital ring homomorphism onto~$\zz$. As a consequence, Proposition~\ref{prop:pullback-bi-HFS} applies. When $m=1$, this recovers the semiring in~\cite[Example~4.6]{GPR24}.
\end{exam}

\medskip
\subsection{A Class of Bi-HF Semidomains from Laurent Polynomials}

Throughout this last section, fix a transcendental number $t \in \rr_{>1}$, and let $x$ be an indeterminate over the field $\qq(t)$. Then let $R$ be the ring of Laurent polynomial $\zz[x^{\pm 1}]$. Then we set
\[
  S_t := \{0\} \cup \big\{ f(t) : f \in \zz[x^{\pm 1}], \ f(t)>0, \ f(1) \in \nn \big\}.
\]
We identify $R$ with its image under evaluation at~$t$. Whenever an irreducible factor of an element of $S_t^*$ is considered, we orient it so that its value at~$t$ is positive and call it \emph{good} or \emph{bad} according as its value at~$1$ is positive or negative. The parity of the bad factors is a concrete two-class block-monoid pattern; the following lemma records its realization inside~$S_t^*$.

\begin{lem}\label{lem:Laurent-bi-HFS-atoms}
  The units of $S_t^*$ are precisely the elements $t^m$ with $m \in \zz$. Up to units, the atoms of $S_t^*$ are precisely the good irreducibles and the products of two bad irreducibles.
\end{lem}

\begin{proof}
  If $f(t)g(t)=1$ with $f(t),g(t) \in S_t^*$, then injectivity of evaluation at the transcendental number~$t$ gives $fg=1$ in~$R$. The units of~$R$ are $\pm x^m$ with $m \in \zz$, and the negative sign is excluded by positivity at both $t$ and~$1$. Hence the units of~$S_t^*$ are exactly the elements $t^m$ with $m \in \zz$. Now take $f(t) \in S_t^*$ and write
  \[
    f=x^m p_1 \cdots p_r
  \]
  in the Laurent UFD~$R$, orienting each nonunit irreducible factor so that $p_i(t)>0$. Since $f(1)>0$, no $p_i(1)$ is zero, and the number of bad factors is even. Every divisor of $f(t)$ in~$S_t^*$ gives, again by injectivity of evaluation at~$t$, a Laurent-polynomial divisor of~$f$. Hence, up to a unit, it is represented by a subproduct of the factors~$p_i$. Such a subproduct belongs to~$S_t^*$ precisely when it contains an even number of bad factors, because all $p_i(t)$ are positive and the units $t^m$ have value~$1$ at~$1$. Thus the divisors of~$f(t)$, modulo units, are identified with the zero-sum submonoid obtained by sending each good factor to~$0$ and each bad factor to the nonzero element of~$C_2$. This is the standard block-monoid picture for zero-sum sequences~\cite[Section~2]{GK21}, and over~$C_2$ its atoms are exactly the single zero terms and the pairs of nonzero terms. Translating back gives the stated classification.
\end{proof}

\begin{thm}\label{thm:Laurent-bi-HFS-counterexample}
  For any transcendental $t \in \rr_{>1}$, the semidomain $S_t$, which is different from $\nn_0$, is a bi-HFS.
\end{thm}

\begin{proof}
  Let $S = S_t$. Then $S$ is a subsemiring of~$\rr_{\ge 0}$ containing $0$ and~$1$: if $f(t),g(t) \in S^*$, then $(f+g)(t)>0$, $(fg)(t)>0$, $(f+g)(1)=f(1)+g(1)>0$, and $(fg)(1)=f(1)g(1)>0$. Also $S \cap \qq=\nn_0$. Indeed, if $f(t)=q \in \qq$, then $f-q$ vanishes at the transcendental number~$t$. After clearing denominators and multiplying by a suitable power of~$x$, this gives the zero polynomial. Thus $f=q$ is constant, and the defining conditions force $q \in \nn_0$.

  For $f(t) \in S^*$, set $\lambda(f(t)):=f(1)$. This map is additive. Every element with $\lambda=1$ is an additive atom, since a nontrivial decomposition would give $1=g(1)+h(1)$ with $g(1),h(1) \in \zz_{>0}$. Conversely, if $f(1)=n>1$, choose $N<0$ so that $(n-1)t^N<f(t)$, and set $a(x):=f(x)-(n-1)x^N$. Then $a(t)>0$ and $a(1)=1$, so $a(t) \in S$ is an additive atom, while $t^N \in S$ is also an additive atom. Hence
  \[
    f(t)=a(t)+\underbrace{t^N+\cdots+t^N}_{n-1 \text{ copies}}.
  \]
  Thus every nonzero element factors additively into atoms, and every additive factorization of~$f(t)$ has length $f(1)$. Hence $S$ is additively half-factorial.

  It remains to prove that $S$ is half-factorial. Factor a nonzero representative $f \in R$ as above. By Lemma~\ref{lem:Laurent-bi-HFS-atoms}, every multiplicative factorization of~$f(t)$ into atoms has length equal to the number of good irreducible factors of~$f$ plus one half the number of bad irreducible factors of~$f$, counted with multiplicity. This length is independent of the factorization, so $S$ is half-factorial. Finally, $t,t^{-1} \in S$, so $t$ is a nontrivial multiplicative unit. Since $t$ is transcendental, $S \neq \nn_0$.
\end{proof}

\bigskip
\section*{Acknowledgments}

During the preparation of this paper, the authors were participants of CrowdMath. They would like to thank MIT PRIMES and the AoPS for providing this forum and stimulating mathematical research. The authors thank Dr.\ Marly Gotti and Dr.\ Harold Polo for carefully proofreading early versions of this paper. Finally, the first author kindly acknowledges partial support from the NSF under award DMS-2213323.

\bigskip

\end{document}